\documentclass{amsart}

\usepackage{graphicx}
\usepackage{amssymb}
\usepackage{epstopdf}
\usepackage{url}
\usepackage[shortlabels]{enumitem}
\usepackage{verbatim}

\usepackage{amsmath,amsthm,amscd,amssymb}
\usepackage{latexsym}

\usepackage{hyperref}
\usepackage{color}
\definecolor{blue}{rgb}{0,0,1}

\definecolor{red}{rgb}{1,0,.2}

\numberwithin{equation}{section}
\theoremstyle{plain}
\newtheorem{theorem}{Theorem}[section]
\newtheorem{thm}[theorem]{Theorem}
\newtheorem{lemma}[theorem]{Lemma}
\newtheorem{corollary}[theorem]{Corollary}
\newtheorem{prop}[theorem]{Proposition}
\newtheorem{conjecture}[theorem]{Conjecture}

\theoremstyle{definition}
\newtheorem{definition}[theorem]{Definition}

\newtheorem{example}[theorem]{Example}

\theoremstyle{remark}
\newtheorem{remark}[theorem]{Remark}

\newtheorem{case[theorem]}{Case}

\def\RR{\mathbb{R}}
\def\ZZ{\mathbb{Z}}
\def\CC{\mathbb{C}}
\def\NN{\mathbb{Z}_{> 0}}

\def\EE{\mathbb{E}}
\def\PP{\mathbb{P}}
\def\supp{\mathrm{supp}}
\def\diam{\mathrm{diam}}
\newcommand{\rbr}[1]{\left( {#1} \right)}

\newcommand{\cbr}[1]{\left\{ {#1} \right\}}

\def\one{\mathbf{1}}

\numberwithin{equation}{section}

\newcommand{\abs}[1]{\lvert#1\rvert}

\begin{document}

\title[Measure and Dimension of Sums and Products]{Measure and Dimension of Sums and Products}

\author{Kyle Hambrook}
\address{San Jose State University, One Washington Square, San Jose, CA 95192}

\author{Krystal Taylor}
\address{The Ohio State University, Columbus, OH 43210}
\thanks{We thank Alex Iosevich and Izabella {\L}aba for their helpful comments.}


\date{}

\begin{abstract}
We investigate the Lebesgue measure, Hausdorff dimension, and Fourier dimension of sets of the form 
 $RY + Z, $ where $R \subseteq (0,\infty)$ and $Y, Z \subseteq \RR^d$. 
We prove a theorem on the Lebesgue measure and Hausdorff dimension of $RY+Z$; 
The theorem is a generalized variant of some theorems of Wolff and Oberlin in which $Y$ is the unit sphere, 
but its proof is much simpler.  
We also prove a deeper existence theorem: For each $\alpha \in [0,1]$ and for each non-empty compact set $R \subseteq (0,\infty)$, 
there exists a compact set $Y \subseteq [1,2]$ such that $\dim_F(Y) = \dim_H(Y) = \overline{\dim_M}(Y) = \alpha$ 
 and $\dim_F(RY) \geq \min\{ 1, \dim_F(R) + \dim_F(Y)\}$.  
 This theorem verifies a weak form of a more general conjecture, and it can be used to produce new Salem sets from old ones.  
\end{abstract}

\maketitle

\section{Introduction}

In this paper, we study the Lebesgue measure, Hausdorff dimension, and Fourier dimension of sets in $\RR^d$ of the form 
\begin{align*}
RY+Z 
= \bigcup_{(r,z) \in R \times Z} (rY+z)   
= \cbr{ry+z : r \in R, y \in Y, z \in Z}. 
\end{align*}
where $R \subseteq (0,\infty)$ and $Y,Z \subseteq \RR^d$ are non-empty sets. 

This problem has been considered before when $Y$ is a smooth surface with non-vanishing curvature; 
see \cite{bourgain-averages, marstrand-packing, mitsis-thesis, oberlin-2006, oberlin-2007, STmeasure, STinterior, stein-maximal-spherical-means, wolff-1997, wolff-smoothing},
In each of these references, the non-vanishing curvature assumption 
is an essential ingredient, as it implies $Y$ supports a Borel probability measure whose Fourier transform decays at $\infty$.  
Our aim is to understand what happens 
when $Y$ is an arbitrary set. 
Fourier decay of measures on $Y$ will turn out to play an important role. 

%
%

We assume throughout that the sets $R$, $Y$, and $Z$ 
are compact in order to guarantee that the set $RY+Z$ is Borel measurable. 
We use $\mathcal{L}_d(A)$, $\dim_H(A)$, $\dim_F(A)$, 
and $\overline{\dim}_M(A)$, 
respectively, to denote the $d$-dimensional Lebesgue measure, Hausdorff dimension, Fourier dimension, 
and upper Minkowski (or box-counting) dimension of a set $A \subseteq \RR^d$. 
Definitions and basic properties of these dimensions are given in Section \ref{dimensions}.  
The expression $a \lesssim b$ means $a \leq C b$ 
for some positive constant $C$ whose precise value is irrelevant in the context. 
The expression $a \approx b$ means $a \lesssim b$ 
and $b \lesssim a$.

\subsection{Main Results}\label{main-results}

This paper has two main results: Theorem \ref{main-thm-1} and Theorem \ref{non-uniform thm sets}. The first motivates the second. 

\begin{thm}\label{main-thm-1}
Let $R \subseteq (0,\infty)$ and $Y, Z \subseteq \RR^d$ be 
non-empty compact sets.  
Let $\delta$ be the maximum of $\dim_F(RY)+\dim_H(Z)$ and $\dim_H(RY)+\dim_F(Z)$. 
\begin{enumerate}[(a)]
\item If $\delta > d$, then $\mathcal{L}_d(RY+Z) > 0$
\item If $\delta \leq d$, then $\dim_{H}(RY+Z) \geq \delta$. 
\end{enumerate}
\end{thm}

This theorem makes precise the intuition that 
the Hausdorff dimension of $RY + Z$ will be at least the 
``total'' dimension of $R$, $Y$, and $Z$, 
and that $RY +Z$ will have positive Lebesgue measure 
whenever the ``total'' dimension exceeds the ambient dimension $d$. 
It is a generalized variant of Theorems of Wolff \cite{wolff-1997, wolff-smoothing} and Oberlin \cite{oberlin-2006,oberlin-2007}, 
which are restricted to the special case where $Y$ is the unit sphere in $\RR^d$. 
A thorough discussion of Theorem \ref{main-thm-1} is delayed 
until Section \ref{lebesgue results}. 
We turn now to motivate our second main result: Theorem \ref{non-uniform thm sets}. 


In light of Theorem \ref{main-thm-1}, lower bounds on $\dim_F(RY)$ are important. 
%
The following proposition gives two general bounds. Part (a) 
is immediate from the definition of Fourier dimension. Part (b) 
is essentially Theorem 7 of Bourgain \cite{bourgain-2010}. 
\begin{prop}\label{basic estimates} \hspace{1pt}
\begin{enumerate}[(a)]
\item If $R \subseteq \RR$ contains a non-zero point and $Y \subseteq \RR^d$, 
then $RY$ contains a 
dilate of $Y$, 
hence $\dim_F(RY) \geq \dim_F(Y)$. 
\item If $R$ and $Y$ are non-empty compact subsets of $(0,\infty)$,  
then 
$\dim_F(RY) \geq \dim_H(R) + \dim_H(Y) - 1$. 
\end{enumerate}
\end{prop}
%
%
%
One might hope for the following estimate, at least for some specific sets $Y$: 
\begin{align}\label{hope}
\dim_F(RY) \geq \min\cbr{d,\dim_F(R) + \dim_F(Y)}.   
\end{align}

%
%
As the following example shows, \eqref{hope} does not hold in general. 
\begin{example}
Any subset of a $(d-1)$-dimensional linear subspace of  
$\RR^d$ has Fourier dimension $0$ 
(see, e.g., \cite[p.41]{Mat15}, \cite[p.64]{wolff-book-2003}). 
Let $Y$ be such a subset. Assume also that $Y$ is compact. 
Let $R \subseteq \mathbb{R}$ be a compact set with $\dim_F(R)>0$. 
Then $RY$ is of the same form as $Y$, 
and so $\dim_F(RY) = \dim_F(Y) = 0 < \dim_F(R)$. 
\end{example}
%
However, if $Y$ equals $S_{d-1}$ (the $(d-1)$-dimensional unit sphere in $\RR^d$), 
then \eqref{hope} does hold: 
\begin{prop}\label{sphere thm sets}
For every non-empty compact set $R \subseteq (0,\infty)$,  
$$\dim_F(RS_{d-1}) \geq \dim_F(R) + \dim_F(S_{d-1}).$$
\end{prop}
Note that 
$\dim_F(S_{d-1}) = \dim_H(S_{d-1}) = \overline{\dim_M}(S_{d-1}) = d-1$ 
(see, e.g., \cite{Mat95, wolff-book-2003}). 
We wonder if there are other sets in $\RR^d$ like the sphere. 
More precisely, we wonder if the following conjecture is true. 

\begin{conjecture}\label{fractal conjecture sets}
For every $\alpha \in [0,d]$, there exists a compact set $Y \subseteq \RR^d$ such that 
$\dim_F (Y) =  \dim_H (Y) = \overline{\dim_M}(Y) = \alpha$ 
and such that,  
for every non-empty compact set $R \subseteq (0,\infty)$, 
\eqref{hope} holds. 
\end{conjecture}

We are not able to prove this conjecture. 
However, if we restrict ourselves to dimension $d=1$ 
and allow $Y$ to depend on $R$, 
we are able to prove the following weaker version of the conjecture, 
which is the second main result of this paper. 
%
%
%
\begin{thm}\label{non-uniform thm sets}
For every $\alpha \in [0,1]$ and every non-empty compact set $R \subseteq (0,\infty)$, 
there exists a compact set $Y \subseteq [1,2]$ 
such that $\dim_F (Y) =  \dim_H (Y) = \overline{\dim_M}(Y) = \alpha$ 
and 
\eqref{hope} holds. 
\end{thm}
The proof of Theorem \ref{non-uniform thm sets} is inspired 
by a construction of Salem sets due to {\L}aba and Pramanik \cite{LP}. 
See also \cite{chen,  hambrook-thesis, HL2013, HL2016}, 
where generalizations of the {\L}aba-Pramanik construction 
were used 
to show the sharpness of fractal Fourier restriction theorems.

As an immediate consequence of Theorems \ref{main-thm-1} and \ref{non-uniform thm sets}, we have: 
\begin{corollary} 
Suppose $Y, Z \subseteq \mathbb{R}$ are compact sets such that $Y$ contains a non-zero point and $\dim_F(Y)+\dim_H(Z) > 0$. 
Let $\alpha \in [0,1]$ such that $\dim_F(Y)+\dim_H(Z) > 1-\alpha$.  
Then there exists a compact set $R \subseteq (0,\infty)$ such that $\dim_H (R) =  \dim_F (R) = \alpha$ and $RY+Z$ has positive Lebesgue measure.   
\end{corollary}


The proofs of Theorem \ref{main-thm-1}, Proposition \ref{sphere thm sets}, 
and Theorem \ref{non-uniform thm sets} are given in 
Sections \ref{main-thm-1 proof}, \ref{sphere thm sets proof}, and \ref{non-uniform thm sets proof}, respectively. 
The following construction will be used in all the proofs. 
\begin{definition}\label{measure defn}
Given a 
Borel measure $\mu$ on $\RR$ and 
a 
Borel measure $\nu$ on $\RR^d$, 
we define the Borel measure $\mu \cdot \nu$ on $\RR^d$ by 
\begin{align*}
\int_{\RR^d} f(z) d (\mu \cdot \nu)(z) 
= \int_{\RR^d} \int_{\RR} f(ry) d\mu(r) d\nu(y). 
\end{align*}
It is readily verified that 
$(\mu \cdot \nu)(\RR^d) = \mu(\RR)\nu(\RR^d)$ and $\supp(\mu \cdot \nu) = \supp(\mu)\supp(\nu)$. 
\end{definition}

\subsection{Dimensions}\label{dimensions}

In this section, we state the definitions and necessary properties of Hausdorff, Fourier, and upper Minkowski dimension. 
For more details, see \cite{falconer-book-fractal-geometry, Mat95, wolff-book-2003}. 
The support of a Borel measure $\mu$ on $\RR^d$, 
denoted $\supp(\mu)$, 
is 
the smallest closed set 
$F$ with $\mu(\RR^d \setminus F) = 0$. 
For each $A \subseteq \RR^d$, let $\mathcal{M}(A)$ 
be the set of all non-zero finite Borel measures 
on $\RR^d$ with compact support contained in $A$. 
For every $\mu \in \mathcal{M}(\RR^d)$ and $0 < s < d$, 
the $s$-energy of $\mu$ is 
\begin{align*}
I_s(\mu) = \int_{\RR^d} \int_{\RR^d} |x-y|^{-s} d\mu(x) d\mu(y) = c(d,s) \int_{\RR^d} |\widehat{\mu}(\xi)|^2 |\xi|^{s-d} d\xi, 
\end{align*}
where $c(d,s) = \pi^{s-d/2} \Gamma((d-s)/2) / \Gamma(s/2)$. 
There are numerous equivalent definitions of Hausdorff dimension. 
For our purposes, the following is most convenient: 
The Hausdorff dimension of a Borel set $A \subseteq \RR^d$ is 
\begin{align}\label{capacity dimension}
\dim_H(A) 
= \sup\cbr{ 0 < s < d : I_s(\mu) < \infty 
\text{ for some } \mu \in \mathcal{M}(A) }. 
\end{align}
The Fourier dimension of a Borel set $A \subseteq \RR^d$ is 
\begin{align}\label{fourier dimension} 
\dim_F(A) 
= 
\sup\cbr{0 < s < d : 
\sup_{\xi \in \RR^d} |\widehat{\mu}(\xi)|^2 |\xi|^s 
< \infty \text{ for some } 
\mu \in \mathcal{M}(A)}. 
\end{align} 
The Fourier dimension of a measure $\mu \in \mathcal{M}(\RR^d)$ is
\begin{align}\label{fourier dimension measure}
\dim_F(\mu) = \sup\cbr{ 0 < s < d: \sup_{\xi \in \RR^d} |\widehat{\mu}(\xi)|^2 |\xi|^s}. 
\end{align}
%
%
%
%
The upper Minkowski dimension of a non-empty bounded set $A \subseteq \RR^d$ is
\begin{align}\label{upper Minkowski dimension}
\overline{\dim_M}(A) 
= \limsup_{\epsilon \to 0^+} 
\frac{\log N(A,\epsilon)}{\log(1 / \epsilon)}. 
\end{align}
where 
$N(A,\epsilon)$ is the smallest number of $\epsilon$-balls needed to cover $A$. 

%
\begin{lemma}\label{dimension properties}
Let $A$ and $B$ be non-empty bounded Borel subsets of $\RR^d$. 
\begin{enumerate}[(a)]
\item If $f:A \to \RR^n$ is a Lipschitz map, then 
$\dim_H(f(A)) \leq \dim_H(A)$. 
\item $\dim_F(A) \leq \dim_H(A) \leq \overline{\dim}_{M}(A)$. 
\item $\dim_H(A) + \dim_H(B) \leq \dim_H(A \times B) \leq \dim_H(A) + \overline{\dim_M}(B)$.  
\item If there is a $\mu \in \mathcal{M}(A)$ 
and positive numbers 
$r_0$ and $s$ 
such that 
\begin{align*}
\mu(B(x,r)) \approx r^s \quad \text{ for all } x \in A, \,\, 0 < r \leq r_0, 
\end{align*}
then $\dim_H(A) = \overline{\dim_M}(A) = s$. 
\end{enumerate}
\end{lemma}

\subsection{ \texorpdfstring{Reverse of \eqref{hope} and Salem Sets}{Reverse of () and Salem Sets}    } \label{equality}

In this section, we find sufficient conditions for the reverse of inequality \eqref{hope} and for $RY$ to be a Salem set. 
A Borel set $A \subseteq \RR^d$ is called Salem if 
$\dim_F(A) = \dim_H(A)$. 
See \cite{LP}, \cite{mattila-1987}, \cite{Mat15}, \cite{mock} 
for some illustrations of the construction of Salem sets and their usefulness in harmonic analysis.

Let $R \subseteq \RR$ and $Y \subseteq \RR^d$ be non-empty compact sets.  
By the definition of Fourier dimension, $\dim_F(RY) \leq d$. 
By parts (a),(b),(c) of 
Lemma \ref{dimension properties} and because 
the map $f(x,y) = xy$ is Lipschitz on bounded sets, 
\begin{align*}
\dim_F(RY) 
\leq \dim_H(RY) 
\leq \dim_H(R \times Y)
\leq \dim_H(R) + \overline{\dim_M}(Y). 
\end{align*}

From this inequality, the following observations are immediate: 
\begin{enumerate}[(i)]
\item If $R$ is a Salem set and $Y$ satisfies $\dim_F(Y) = \dim_H(Y) = \overline{\dim_M}(Y)$, 
then the reverse of \eqref{hope} holds, i.e., 
$$
\dim_F(RY) \leq \min\cbr{d,\dim_F(R) + \dim_F(Y)}. 
$$
\item If, in addition to the hypotheses of (i), \eqref{hope} holds, then $RY$ is a Salem set. 
\end{enumerate}

The hypotheses of (i) and (ii) are satisfied 
if, for example, $Y$ is the unit sphere $S_{d-1}$ or a set furnished by Theorem \ref{non-uniform thm sets}, 
The construction of Salem sets is generally non-trivial; (ii) gives us a way to produce new Salem sets from old ones.

\subsection{Discussion of Theorem \ref{main-thm-1}}\label{lebesgue results}


Let $S_{d-1}$ denote the $(d-1)$-dimensional unit sphere in $\RR^d$. 
Theorem \ref{main-thm-1} is a generalized variant of 
the following theorems of Wolff and Oberlin.  

  


\begin{thm}[Wolff, Oberlin]\label{real wolff thm}
Let $K \subseteq (0,\infty) \times \RR^d$ 
be a non-empty compact set.  
If $\dim_H(K) > 1$, then 
$$
\mathcal{L}_d\rbr{\bigcup_{(r,z) \in K} (rS_{d-1}+z)} > 0. 
$$
\end{thm}


\begin{thm}[Wolff]\label{haus wolff thm}
Let $K \subseteq (0,\infty) \times \RR^d$ be a non-empty compact set.   
Let $C \subseteq \RR^d$ be the set of centers of the spheres 
$\cbr{rS_{d-1}+z : (r,z) \in K}$, i.e., 
$C = \cbr{z: (r,z) \in K \text{ for some } r }$. 
If $\dim_H(K) \leq 1$ and $\dim_H( C ) \leq 1$, then 
\begin{align*}
\dim_H \rbr{ \bigcup_{(r,z) \in K} (rS_{d-1}+z) } 
\geq \dim_H(C) + d-1. 
\end{align*} 
\end{thm}

%
\begin{thm}[Oberlin]\label{haus oberlin thm} 
Let $K \subseteq (0,\infty) \times \RR^d$ 
be a non-empty compact set.  
If $\dim_H(K) \leq 1$ and $\dim_H(K) <  (d-1)/2$, 
then 
\begin{align*}
\dim_H \rbr{ \bigcup_{(r,z) \in K} (rS_{d-1}+z) } 
\geq
\dim_H(K) + d-1. 
\end{align*}
\end{thm}


Theorem \ref{real wolff thm} is due to 
Wolff \cite[Corollary 3]{wolff-smoothing}  when $d=2$ 
and Oberlin \cite[Corollary 1]{oberlin-2006} when $d \geq 3$. 
Theorem \ref{haus oberlin thm} is due to 
Oberlin \cite[Theorem $3_S$]{oberlin-2007}. 
Theorem \ref{haus wolff thm} is due to Wolff \cite[Corollary 5.4]{wolff-1997} (see also the remark following the proof). 
In fact, Wolff and Oberlin proved slightly more general results;  
we have stated special cases to simplify comparison 
to Theorem \ref{main-thm-1}. 

Wolff obtained Theorem \ref{real wolff thm} for 
$d = 2$ 
as a corollary of localized $L^p$ estimates on functions 
with Fourier support near the light cone;  
the proof 
involves intricate bounds on circle tangencies. 
For $d=2$, the proof of Theorem \ref{haus wolff thm} is based on an 
$L^3$-$L^3$ circular maximal inequality, 
whose proof in turn is also based on bounds on circle tangencies. 
For $d \geq 3$, Theorem \ref{haus wolff thm} follows from an 
an easier $L^2$-$L^2$ spherical maximal inequality. 
Oberlin obtained Theorem 
\ref{real wolff thm} for $d \geq 3$ and Theorem \ref{haus oberlin thm} by way of 
estimates for spherical averaging operators. 
Mitsis \cite{mitsis-thesis} previously proved 
a special case of Theorem \ref{real wolff thm} with an additional hypothesis 
on the Hausdorff dimension of the set of centers 
$C = \cbr{ z: (r,z) \in K \text{ for some } r }$; 
the methods are similar to those used 
by Wolff for Theorem \ref{haus wolff thm}. 
An alternative proof of Mitsis' result using the technology of 
spherical maximal operators in a fractal setting follows as an 
immediate consequence of the work of 
Krause, Iosevich, Sawyer, Taylor, and Uriarte-Tuero \cite{Maximal}.

In contrast to the proofs of Theorems \ref{real wolff thm}, 
\ref{haus wolff thm}, \ref{haus oberlin thm}, 
our proof of Theorem \ref{main-thm-1} 
is extremely short and uses only elementary 
Fourier analysis and geometric measure theory. 
As described in Section \ref{main-results}, 
Theorem \ref{main-thm-1} also leads to the problem of  
of obtaining lower bounds 
on the Fourier dimension of the Minkowski product of two sets, 
which motivates Theorem \ref{non-uniform thm sets}.


We call Theorem \ref{main-thm-1} a generalized variant of 
Theorems 
\ref{real wolff thm},  \ref{haus wolff thm}, \ref{haus oberlin thm}, 
because 
Theorem \ref{main-thm-1} allows an arbitrary set 
$Y$ in place of the sphere $S_{d-1}$. 
However, Theorem \ref{main-thm-1} 
is not a true generalization. 
For one thing, the union 
$RY+Z = \bigcup_{(r,z) \in R \times Z } (rY+Z)$  
in Theorem \ref{main-thm-1} 
is over the Cartesian product $R \times Z$, 
while the union in Theorems 
\ref{real wolff thm},  \ref{haus wolff thm}, \ref{haus oberlin thm}, 
is over a more general set $K$. 

To compare Theorem \ref{main-thm-1} to Theorems 
\ref{real wolff thm},  \ref{haus wolff thm}, \ref{haus oberlin thm} 
on the same footing, 
let $Y=S_{d-1}$, 
let $K = R \times Z$ 
(where $R \subseteq (0,\infty)$ and $Z \subseteq \RR^d$ are compact sets), 
and let $\delta$ be as in Theorem \ref{main-thm-1}. 
Arguing as in Section \ref{equality}, 
we find  
%
%
\begin{align}\label{inequality}
\dim_H(K) + d-1 
= \dim_H(R \times S_{d-1} \times Z)
\geq 
\dim_H(R S_{d-1}) + \dim_H(Z)
\geq \delta.  
\end{align}
%
%
From \eqref{inequality}, we see that the hypothesis of 
Theorem \ref{main-thm-1}(a) ($\delta > d$) 
is stronger than the hypothesis of 
Theorem \ref{real wolff thm} ($\dim_H(K) > 1$), 
while the conclusions are the same. 
Thus
Theorem \ref{real wolff thm} is stronger than Theorem \ref{main-thm-1}(a).  
The situation is more complicated when comparing 
Theorem \ref{main-thm-1}(b) 
to Theorems \ref{haus wolff thm} and \ref{haus oberlin thm}. 
By \eqref{inequality}, the hypothesis of Theorem \ref{main-thm-1}(b) 
is weaker than the hypotheses of Theorems \ref{haus wolff thm} 
and \ref{haus oberlin thm}, 
but the conclusion of Theorem \ref{main-thm-1}(b) is also weaker.  



Theorem \ref{main-thm-1}(b) fills some gaps in 
Theorems \ref{haus wolff thm} and \ref{haus oberlin thm}.  
In particular, 
Theorem \ref{haus oberlin thm} misses the endpoint case 
$\dim_H(K)=1$ when $d=3$ 
and misses the interval $1/2 \leq \dim_H(K) \leq 1$ when $d=2$. 
Meanwhile, Theorem \ref{haus wolff thm} can only establish that 
$\dim_H(\bigcup_{(r,z) \in K} rS_{d}+z)$ is at least $\dim_H(C) + d-1$, 
rather than at least $\dim_H(K) + d-1$.  
As the next example shows, 
Theorem \ref{main-thm-1}(b) can give superior information in these cases. 

\begin{example}\label{gaps filled 1}
Let $d=3$. 
Let 
$A$ be a compact subset of $(0,\infty)$ 
such that 
$\dim_F(A)=\dim_H(A)=\overline{\dim_M}(A) = 1/4$ 
(Theorem \ref{non-uniform thm sets} furnishes such a set). 
Let $R=A$, $Z = A^3$, and $K = R \times Z$.  
By Lemma \ref{dimension properties}(c), $\dim_H(Z) = 3/4$ and $\dim_H(K) = 1$. 
Since $\dim_H(K) \geq (d-1)/2$, 
Theorem \ref{haus oberlin thm} gives no information 
about the Hausdorff dimension of 
$\bigcup_{(r,z) \in K} (rS_{2}+z)$. 
Since $Z = C = \cbr{z : (r,z) \in K \text{ for some } r }$, 
Theorem \ref{haus wolff thm} implies 
only that the Hausdorff dimension of $\bigcup_{(r,z) \in K} (rS_{2}+z)$ is 
at least $2 + 3/4$.  
However, by Proposition \ref{sphere thm sets}, 
$
\dim_F(RS_2) + \dim_H(Z) \geq \dim_{F}(R)+\dim_F(S_{2})+\dim_H(Z) = 3. 
$
So Theorem \ref{main-thm-1}(b) implies 
the Hausdorff dimension of 
$\bigcup_{(r,z) \in K} (rS_{2}+z) = R S_{2} + Z$ equals $3$. 
(Taking $R=A$ and $Z=A^2$ gives a similar example when $d=2$.) 
\end{example}

Theorem \ref{real wolff thm} 
can be generalized to any 
$(d-1)$-dimensional surface 
with non-vanishing Gaussian curvature 
\cite{alex}. 
However, there is no hope of replacing the sphere in 
Theorem \ref{real wolff thm} by an arbitrary 
$(d-1)$-dimensional set, as the following example shows.   

\begin{example}\label{cone example}
Let $Y$ be the intersection of a $(d-1)$-dimensional cone through the origin 
with an open spherical shell centered at the origin (such as the set $\cbr{x \in \RR^d: 1/2 < |x| < 2}$).  
Then $RY$ is of the same form whenever $R \subseteq (0,\infty)$ 
contains a non-zero point. 
Let $R = [a,b] \subseteq (0,\infty)$ be a compact interval, 
let $Z \subseteq \RR^d$ be a compact set with $\dim_H(Z) \in (0,1)$, 
and let $K = R \times Z$.  
By Lemma \ref{dimension properties}(c), 
$\dim_H(K) \geq \dim_H(R) + \dim_H(Z) > 1$. 
But $\bigcup_{(r,z) \in K} (rY+z) = RY+Z$ 
has Hausdorff dimension $\leq d-1 + \dim_H(Z) < d$, 
hence has $d$-dimensional Lebesgue measure $0$. 
To see that $\dim_H(RY+Z) \leq d-1 + \dim_H(Z)$, 
we first observe that, as $RY$ is a smooth $(d-1)$-dimensional manifold, 
its Hausdorff and upper Minkowski dimensions are both equal to $d-1$. 
(To see this, one can appeal to Lemma \ref{dimension properties}(d) and take $\mu$ 
to be the restriction of the $(d-1)$-dimensional Hausdorff measure to $RY$). 
Thus, by Lemma \ref{dimension properties}(c), 
$\dim_H(RY \times Z) = d-1 +\dim_H(Z)$. 
Then, since $f(x,y) = x+y$ is Lipschitz, 
Lemma \ref{dimension properties}(a) implies 
$\dim_H(RY + Z) \leq d-1 +\dim_H(Z)$.  
\end{example}


\section{Proof of Theorem \ref{main-thm-1}}\label{main-thm-1 proof}

\subsection{Proof of Part (a)}

Assume $\dim_F(RY)+\dim_H(Z) \geq \dim_H(RY)+\dim_F(Z)$. 
The proof in the opposite case is similar. 
Assume $\dim_F(RY)+\dim_H(Z) > d$.  
Choose $0 < \alpha < \dim_F(RY)$ and $0 < \beta < \dim_H(Z)$ 
such that $\alpha+\beta = d$. 
Choose $\mu \in \mathcal{M}(RY)$ and $\nu \in \mathcal{M}(Z)$ 
such that 
$\sup_{\xi \in \RR^d} |\widehat{\mu}(\xi)|^2 |\xi|^{\alpha} < \infty$  
and $I_{\beta}(\nu) < \infty$. 
Then $\mu \ast \nu \in \mathcal{M}(RY+Z)$ and 
\begin{align*}
\int |\widehat{\mu \ast \nu}(\xi)|^2 d\xi 
= \int |\widehat{\mu}(\xi)|^2 |\widehat{\nu}(\xi)|^2 d\xi 
\lesssim \int |\xi|^{(d-\alpha) - d} |\widehat{\nu}(\xi)|^2 d\xi 
= I_{\beta}(\nu) 
< \infty. 
\end{align*} 
Since $\widehat{\mu \ast \nu}$ 
is in $L^2$,  
a standard argument (e.g. see \cite[Theorem 3.3]{Mat15}) 
shows $\mu \ast \nu$ is absolutely continuous with $L^2$ density, 
hence $\mathcal{L}_d(RY+Z) > 0$. 

\subsection{Proof of Part (b)}

Assume $\dim_F(RY)+\dim_H(Z) \geq \dim_H(RY)+\dim_F(Z)$. The proof in the opposite case is similar. 
Let $s \in (0,d)$ with $\dim_F(RY)+\dim_H(Z) > s$.  
If either $\dim_F(RY) = 0$ or $\dim_H(Z)=0$, then the result is immediate; assume otherwise.  
Choose $0 < \alpha < \dim_F(RY)$ and $0 < \beta < \dim_H(Z)$ such that $\alpha+\beta = s$. 
Choose $\mu \in \mathcal{M}(RY)$ and $\nu \in \mathcal{M}(Z)$ 
such that 
$\sup_{\xi \in \RR^d} |\widehat{\mu}(\xi)|^2 |\xi|^{\alpha} < \infty$  
and $I_{\beta}(\nu) < \infty$. 
Then $\mu \ast \nu \in \mathcal{M}(RY+Z)$ 
and 
\begin{align*}
I_{s}(\mu \ast \nu)
&= \int |\widehat{\mu \ast \nu}(\xi)|^2 |\xi|^{s-d} d\xi 
= \int |\widehat{\mu}(\xi)|^2 |\widehat{\nu}(\xi)|^2 |\xi|^{s-d} d\xi 
\\
&\lesssim \int |\xi|^{(s-\alpha)-d} |\widehat{\nu}(\xi)|^2 d\xi 
= I_{\beta}(\nu) 
< \infty. 
\end{align*} 
By \eqref{capacity dimension} and our choice of $s$, 
$\dim_H(RY+Z) \geq \dim_F(RY)+\dim_H(Z)$.

\section{Proof of Proposition \ref{sphere thm sets}}\label{sphere thm sets proof}

Let $\sigma$ be the surface measure on the $(d-1)$-dimensional unit sphere $S_{d-1} \subseteq \RR^d$. 
The following asymptotic is well-known (e.g., see \cite{Mat15, wolff-book-2003}):  
For 
$\xi \in \RR^d$, 
\begin{align*}
\notag
\widehat{\sigma}(\xi) 
&= 
2|\xi|^{-(d-1)/2}\cos\rbr{ 2\pi \rbr{|\xi| - \frac{d-1}{8}}  } 
+ O\rbr{|\xi|^{-(d+1)/2}} \\
&=
  |\xi|^{-(d-1)/2} \rbr{  c_d e^{2\pi i |\xi| } 
  + \overline{c_d} e^{-2 \pi i |\xi|} } + O\rbr{|\xi|^{-(d+1)/2}},
\end{align*}
where $c_d = e^{-\pi i (d-1)/4}$. Let $R \subseteq (0,\infty)$.  
Let $\mu \in \mathcal{M}(R)$ be arbitrary. 
Choose $a,b>0$ such that $\supp(\mu) \subseteq [a,b] \subseteq (0,\infty)$. 
Define $\mu_0$ by $d\mu_0(r) = r^{\frac{(d-1)}{2}}  d \mu(r)$. 
Then $\supp(\mu_0) = \supp(\mu)$, $\mu_0 \in \mathcal{M}(R)$, 
and $\mu_0 \cdot \sigma \in \mathcal{M}(RS_{d-1})$.
Furthermore, 
\begin{align*}
\widehat{\mu_0 \cdot \sigma}(\xi) 
= \int_a^b \widehat{\sigma}(r\xi) d\mu_0(r) 
%
= 
|\xi|^{-(d-1)/2} \rbr{
c_d  \widehat{\mu}(-|\xi|)
+ \overline{c_d} \widehat{\mu}(|\xi|)
}
+ O\rbr{|\xi|^{-(d+1)/2}}
\end{align*}
for all sufficiently large $\xi \in \RR^d$. 
Therefore 
$\dim_F(RS_{d-1}) \geq \dim_F(\mu_0 \cdot \sigma) \geq \dim_F(\mu) + d-1$ 
and (consequently) 
$\dim_F(RS_{d-1}) \geq \dim_F(R) + d-1$. 

\section{Proof of Theorem \ref{non-uniform thm sets}}
\label{non-uniform thm sets proof}

In light of parts (b) and (d) of Lemma \ref{dimension properties} 
and the definition of Fourier dimension, 
Theorem \ref{non-uniform thm sets} 
is implied by:

\begin{thm}\label{non-uniform thm measures}
Let $\alpha \in [0,1]$. 
Let $R$ be a non-empty compact subset of $(0,\infty)$. 
Let $\nu \in \mathcal{M}(R)$. 
Then there is a Borel probability measure $\mu$ on $\RR$ 
such that 
\begin{enumerate}[(a)]
\item $\supp(\mu) \subseteq [1,2]$,
\item $\mu(B(x,r)) \approx r^{\alpha}$ for all $x \in \supp(\mu)$ and $0 < r \leq 1$,  
\item $\dim_F(\mu) \geq \alpha$, 
\item $\dim_F(\mu \cdot \nu) \geq \min\cbr{1,\dim_F(\mu) + \dim_F(\nu)}$.
\end{enumerate}
\end{thm}

The following two sections are devoted to proving 
Theorem \ref{non-uniform thm measures}. In the first section, we give a general construction of a measure $\mu$ that satisfies parts (a) and (b) of Theorem \ref{non-uniform thm measures}. 
Nothing in the first section will depend on $\nu$. In the second section, we specialize the construction of $\mu$ to prove parts (c) and (d). The specialization will depend on $\nu$. 

Before we begin, we dispense with a trivial case: If $\alpha = 0$, 
then taking $\mu$ to be a point mass gives the desired result. 
Hereafter, we assume $\alpha \in (0,1]$.

\subsection{Proof of Theorem \ref{non-uniform thm measures}: General Construction}



For every $n \in \NN$, we use the notation $[n] = \cbr{0,1,\ldots,n-1}$  For sequences $(t_j)_{j=1}^{\infty}$ and $(n_j)_{j=1}^{\infty}$ of positive integers, we use the notation $T_j = t_1 \cdots t_j$ and $N_j = n_1 \cdots n_j$. We also use the empty product convention, so that $T_0 = N_0 = 1$.

Fix an integer $n_{\ast} \geq 2$. 
Fix sequences $(t_j)_{j=1}^{\infty}$ and 
$(n_j)_{j=1}^{\infty}$ of positive integers such that, 
for all $j \in \NN$, we have $2 \leq n_j \leq n_{\ast}$, 
$1 \leq t_j \leq n_j$, and $T_j \approx N_j^{\alpha}$.

We recursively define two families of sets: 
$$
\cbr{A_j : j \in \ZZ_{\geq 0}} \text{ and } \cbr{B_{j+1,a} : j \in \ZZ_{\geq 0}, a \in A_j}. 
$$
Define $A_{0} = \cbr{1}$. 
Assuming that $A_j$ has been defined for a fixed $j \in \ZZ_{\geq 0}$, for each $a \in A_{j}$ choose a set $B_{j+1,a} \subseteq N_{j+1}^{-1}[n_{j+1}]$ such that $|B_{j+1,a}| = t_{j+1}$. 
Later we make a specific choice for the sets $B_{j+1,a}$, but for now the sets $B_{j+1,a}$ are arbitrary. 
Define $A_{j+1} = \bigcup_{a \in A_j} \rbr{a+B_{j+1,a}}$. 
Note that this recursive definition implies that $A_j \subseteq [1,2)$ and $|A_j| = T_j$ for all $j \in \ZZ_{\geq 0}$. 

Question: What are the sets $A_j$ and $B_{j+1,a}$? Answer: They are sets of endpoints in the following Cantor set construction. 
Start with the interval $[1,2]$. Divide it into $n_1$ intervals of length $1/N_1$, keep $t_1$ of these intervals, and discard the rest. For each of the kept intervals, we do the following: Divide it into $n_2$ intervals of length $1/N_2$, keep $t_2$ of these intervals, and discard the rest. This gives, in total, $T_2$ intervals of length $1/N_2$. Continuing in this way, at the $j$-th stage we have $T_j$ intervals of length $1/N_j$. The set of left endpoints of these intervals is $A_j$. For each of these intervals, we do the following: Divide it into $n_{j+1}$ intervals of length $1/N_{j+1}$, keep $t_{j+1}$ of these intervals, and discard the rest. 
If 
$a$  
is the left endpoint of the interval we started with, the set of left endpoints of the intervals kept is $a + B_{j+1,a}$. 
The union of all the sets $a + B_{j+1,a}$ (as $a$ ranges over $A_{j}$) is $A_{j+1}$. The Cantor set constructed is 
\begin{align}\label{cantor set}
\bigcap_{j=0}^{\infty} \bigcup_{a \in A_j} [a,a+N_j^{-1}]. 
\end{align}



For each $j \in \ZZ_{\geq 0}$, define $\mu_j$ to be the probability measure whose density with respect to Lebesgue measure on $\RR$ is 
$$
\mu_j = \frac{N_j}{T_j} \sum_{a \in A_j} \one_{[a,a+N_{j}^{-1}]}. 
$$
Note that we have abused notation by using the same symbol for a measure and its density; we will continue to do this.  
For each $j \in \ZZ_{\geq 0}$, 
\begin{align}\label{supp mu j formula}
\supp(\mu_j) = \bigcup_{a \in A_j} [a,a+N_{j}^{-1}]. 
\end{align}
Moreover,  for every $j,k \in \ZZ_{\geq 0}$ with $j \leq k$ and for every $a \in A_j$, 
\begin{align}\label{mu k formula}
\mu_k([a,a+N_{j}^{-1}]) = T_j^{-1}. 
\end{align}

\begin{lemma}\label{weak convergence}
The sequence $(\mu_j)_{j=0}^{\infty}$ converges weakly (i.e., in distribution) to a probability measure $\mu$. 
\end{lemma}
\begin{proof}
For each $j \in \ZZ_{\geq 0}$, let $F_j$ be the cumulative distribution function of $\mu_j$, i.e., $F_j(t)=\mu_j((-\infty,t])$ for all $t \in \RR$. 
Let $t \in \RR$. 
If $t \leq \min A_j$, then $F_{j}(t) = F_{j+1}(t) = 0$.  
Now assume $t \geq \min A_j$. 
Let $a(t)$ be the unique element of $A_j$ such that $a(t) \leq t < a(t)+N^{-j}$. 
Since $F_{j}(a) = F_{j+1}(a)$ for each $a \in A_j$ and because of \eqref{mu k formula}, we have 
\begin{align*}
|F_{j+1}(t) - F_{j}(t)| 
&\leq |\mu_{j+1}((a(t),t]) - \mu_{j}((a(t),t])| \\
&\leq \mu_{j+1}([a(t),a(t)+N_j^{-1}]) + \mu_{j}([a(t),a(t)+N_j^{-1}]) \\
&= 2T_{j}^{-1} \approx N_{j}^{-\alpha} \leq 2^{-j\alpha}. 
\end{align*}
%
%
It follows that $(F_j)_{j=0}^{\infty}$ is a uniformly Cauchy and (consequently) uniformly convergent sequence of continuous cumulative distribution functions. 
Therefore the limit $F$ is a continuous cumulative distribution function of a Borel probability measure $\mu$ on $\RR$. 
Hence $(\mu_j)_{j=0}^{\infty}$ converges weakly to $\mu$.   
%
%
%
%
\end{proof}

Combining \eqref{mu k formula} with Lemma \ref{weak convergence} gives, 
for every $j \in \ZZ_{\geq 0}$ and $a \in A_j$,   
\begin{align}\label{mu formula}
\mu([a,a+N_{j}^{-1}]) = T_j^{-1}. 
\end{align}

\begin{lemma}\label{support lemma}
The support of $\mu$ is 
\begin{align*}
\supp(\mu) = \bigcap_{j=0}^{\infty} \supp(\mu_j) = \bigcap_{j=0}^{\infty} \bigcup_{a \in A_j} [a,a+N_j^{-1}]. 
\end{align*}
\end{lemma}
\begin{proof}
The second equality is immediate from \eqref{supp mu j formula}. 
For the first equality, we consider $\subseteq$ and $\supseteq$ separately. 

$\subseteq:$ We prove the contrapositive. Suppose $x \in \RR \setminus \supp(\mu_{j_0})$ for some $j_0 \in \ZZ_{\geq 0}$. Since $\RR \setminus \supp(\mu_{j_0})$ is open, there is an open ball $B(x,r_0)$ contained in $\RR \setminus \supp(\mu_{j_0})$. Since $(\supp(\mu_j))_{j=0}^{\infty}$ is a decreasing sequence of sets, $B(x,r_0)$ is contained in $\RR \setminus \supp(\mu_{j})$ for every $j \geq j_0$. Thus $\mu_j(B(x,r_0))=0$ for every $j \geq j_0$. Choose a continuous function $\phi:\RR \to \RR$ such that $\one_{B(x,r_0/2)} \leq \phi \leq \one_{B(x,r_0)}$. Then, since $\mu_{j} \to \mu$ weakly, we have $\mu(B(x,r_0/2)) \leq \int \phi d\mu = \lim_{j \to \infty} \int \phi d\mu_j \leq \lim_{j \to \infty} \mu_j(B(x,r_0)) = 0$. So $\mu(B(x,r_0/2))=0$. It follows that $x \in \RR \setminus \supp(\mu)$. 

$\supseteq:$ Let $x \in \bigcap_{j=0}^{\infty} \supp(\mu_j)$. Let $B(x,r)$ be any open ball centered at $x$. Choose $j$ large enough that $N_{j}^{-1} < r/2$. Since $x \in \supp(\mu_j)$, we have $\mu_j(B(x,r/2)) > 0$, hence $B(x,r/2)$ intersects $[a,a+N_{j}^{-1}]$ for some $a \in A_j$. Therefore $B(x,r)$ contains $[a,a+N_{j}^{-1}]$. Choose a continuous function $\phi:\RR \to \RR$ such that $\one_{[a,a+N_{j}^{-1}]} \leq \phi \leq \one_{B(x,r)}$. By \eqref{mu k formula}, $T_j^{-1} = \mu_k([a,a+N_{j}^{-1}]) \leq \int \phi d\mu_k$ for all $k \geq j$. Then, since $\mu_{k} \to \mu$ weakly, we have $T_j^{-1} \leq \lim_{k \to \infty} \int \phi d\mu_k = \int \phi d\mu \leq \mu(B(x,r))$. So $\mu(B(x,r)) > 0$.  Since $B(x,r)$ was arbitrary, it follows that $x \in \supp(\mu)$. 
\end{proof}

The measure $\mu$ is the so-called natural measure on the Cantor set \eqref{cantor set}. 

\begin{lemma}\label{Hausdorff mu}
For every interval $I$ with diameter $|I|$, 
$
\mu(I) \lesssim |I|^{\alpha}.   
$
\end{lemma}
\begin{proof}
%
If $|I| > 1$, then
$$
\mu(I) \leq \mu(\RR) = 1 \leq |I|^{\alpha} 
$$
since $\mu$ is a probability measure. 
Now suppose $|I| \leq 1$. 
Choose $j_0 \in \ZZ_{\geq 0}$ such that $N_{j_+1}^{-1} \leq |I| \leq N_{j_0}^{-1}$. 
Assume $I$ intersects $\text{supp}(\mu)$ (otherwise $\mu(I) = 0$). 
Then $I$ intersects an interval $[a,a+N_{j_0}^{-1}]$ for some $a \in A_{j_0}$. 
Since $|I| \leq N_{j_0}^{-1}$ and $A_j \subseteq N_{j_0}^{-1}[N_{j_0}]$, 
there are at most two such intervals; call them $J_1$ and $J_2$. 
By \eqref{mu formula}, the choice of the sequences $(n_j)$ and $(t_j)$, and the choice of $j_0$, 
we have 
\begin{align*}
\mu(I) 
= \mu(I \cap J_1) + \mu(I \cap J_2) 
\leq \mu(J_1) + \mu(J_2)
= \dfrac{2}{T_{j_0}}  
\approx \frac{1}{N_{j_0 +1}^{\alpha}} 
\leq |I|^{\alpha} 
\end{align*}
\end{proof}

\begin{lemma}\label{lower regularity}
For every interval $I$ with center in $E$ and diameter $|I| \leq 1$, 
$
\mu(I) \gtrsim |I|^{\alpha}. 
$
\end{lemma}
\begin{proof}
Choose ${j_0} \in \ZZ_{> 0}$ such that 
$2N_{{j_0}+1}^{-1} \leq |I| \leq 2N_{{j_0}}^{-1}$. 
Since the center of $I$ belongs to  
$E = \bigcap_{j=0}^{\infty} \bigcup_{a \in A_j} [a,a+N_j^{-1}]$, 
it belongs to the interval $[a,a+N_{{j_0}+1}^{-1}]$ for some $a \in A_{{j_0}+1}$. 
Then, since $2N_{{j_0}+1}^{-1} \leq |I|$, $[a,a+N_{{j_0}+1}^{-1}] \subseteq I$. 
By \eqref{mu formula}, the choice of the sequences $(n_j)$ and $(t_j)$, 
and the choice of $j_0$, 
we have 
\begin{align*}
\mu(I) \geq \mu([a,a+N_{{j_0}+1}^{-1}]) 
= \frac{1}{T_{{j_0}+1}} 
\approx \frac{1}{N_{{j_0}+1}^{\alpha}} 
=  \frac{1}{2^{\alpha} n_{{j_0}+1}^{\alpha}} \frac{2^{\alpha}}{N_{{j_0}}^{\alpha}} 
\geq \frac{1}{2^{\alpha} n_{*}^{\alpha}} |I|^{\alpha}. 
\end{align*}
\end{proof}


\subsection{Proof of Theorem \ref{non-uniform thm measures}: Fourier Decay}


We will now prove that the sets $B_{j+1,a}$ 
can be chosen so that 
(c) and (d) of Theorem \ref{non-uniform thm measures} hold. 
Before beginning, we outline the proof and make two remarks. 

The idea of the proof is to choose the sets $B_{j+1,a}$ 
using the probabilistic method. 
More specifically, for a given $j$ and a given set $A_j$ of numbers $a$, 
we choose each $B_{j+1,a}$ uniformly at random from a finite collection of possible sets. 
Thus the differences (i) $\widehat{\mu_{j+1}}(\xi) - \widehat{\mu_j}(\xi)$ 
and (ii) $\widehat{\mu_{j+1} \cdot \nu}(\xi) - \widehat{\mu_j \cdot \nu}(\xi)$ 
can be written as sums of finitely many independent random variables. 
We use Hoeffding's large deviation inequality 
to show that, with positive probability,  
the differences (i) and (ii) are small. 
We deduce that there exists a choice 
of the sets $B_{j+1,a}$ which makes 
(i) and (ii) small. 
This probabilistic argument is the heart of the proof. 
It appears as Lemma \ref{choice tele lemma} below. 
After Lemma \ref{choice tele lemma}, 
we use telescoping sum and geometric series arguments to 
deduce the desired decay estimates for $\widehat{\mu}$ and 
$\widehat{\mu \cdot \nu}$. 
Before getting to Lemma \ref{choice tele lemma}, 
we give several preparatory results. 


\begin{remark}
The sets $B_{j+1,a}$ depend on $\nu$ 
because they are chosen (probabilistically) to make  
$\widehat{\mu_{j+1} \cdot \nu}(\xi) - \widehat{\mu_j \cdot \nu}(\xi)$ small. 
%
It would be possible to work 
with a finite or countably infinite collection of measures 
$\nu_1$, $\nu_2$, $\ldots$ and choose the sets $B_{j+1,a}$ so that 
$\widehat{\mu_{j+1} \cdot \nu_i}(\xi) - \widehat{\mu_j \cdot \nu_i}(\xi)$ 
is small for each $i$. 
(As an exercise, we invite the reader to verify this by modifying the proof of Lemma \ref{choice tele lemma}). 
Countable subadditivity of the probability measure is what makes this possible. 
However, we do not know how to choose the sets $B_{j+1,a}$ so that 
$\widehat{\mu_{j+1} \cdot \nu}(\xi) - \widehat{\mu_j \cdot \nu}(\xi)$ 
is small for all measures $\nu$. If we could do that, 
then we could make the measure $\mu$ independent of $\nu$ 
in Theorem \ref{non-uniform thm measures}, 
and hence make the set $Y$ independent of $R$ in Theorem \ref{non-uniform thm sets}. 
\end{remark}

Now we begin the proof. 
Let $j \in \ZZ_{\geq 0}$. 
We write the densities of $\mu_j$ and $\mu_{j+1}$ in more convenient forms. 
By partitioning $[0,N_{j}^{-1}]$ into intervals of length $N_{j+1}^{-1}$, we see that 
\begin{align*}
\mu_j = \left(\frac{N_{j+1}}{T_{j}}\right) \left(  \frac{1}{n_{j+1}}\right)
\sum_{a \in A_j} \sum_{b \in N_{j+1}^{-1}[n_{j+1}]} \one_{a+b+[0,N_{j+1}^{-1}]}.
\end{align*}
Since $A_{j+1} = \bigcup_{a \in A_j} \rbr{a + B_{j+1,a}}$, we also have 
\begin{align*}
\mu_{j+1} = \left(\frac{N_{j+1}}{T_{j}}\right) \left(  \frac{1}{t_{j+1}}\right)
 \sum_{a \in A_j} \sum_{b \in B_{j+1,a}} \one_{a+b+[0,N_{j+1}^{-1}]}.
\end{align*}
For each $a \in A_j$, $b \in B_{j+1,a}$, and $\xi \in \RR$, define 
\begin{align}
\label{I defn}
I(a,b,j,\xi) &= \int_{[0,1]} e^{-2 \pi i (\xi / N_{j+1})(aN_{j+1}+bN_{j+1}+x)} dx, \\
\label{J defn}
J(a,b,j,\xi) &= \int_{\RR} \int_{[0,1]} e^{-2 \pi i (\xi / N_{j+1})(aN_{j+1}+bN_{j+1}+x)y} dx d\nu(y), \\
\label{X defn}
X_a(j,\xi) &= \frac{1}{t_{j+1}} \sum_{b \in B_{j+1,a}} I(a,b,j,\xi) - \frac{1}{n_{j+1}} \sum_{b \in N_{j+1}^{-1}[n_{j+1}]} I(a,b,j,\xi), \\
\label{Y defn}
Y_a(j,\xi) &= \frac{1}{t_{j+1}} \sum_{b \in B_{j+1,a}} J(a,b,j,\xi) - \frac{1}{n_{j+1}} \sum_{b \in N_{j+1}^{-1}[n_{j+1}]} J(a,b,j,\xi). 
\end{align}
%
It follows that, for all $\xi \in \RR$,  
\begin{align}
\label{mu equation}
\widehat{\mu_{j+1}}(\xi) - \widehat{\mu_{j}}(\xi) = \frac{1}{T_j} \sum_{a \in A_j} X_a(j,\xi) \\
\label{m equation}
\widehat{\mu_{j+1} \nu}(\xi) - \widehat{\mu_{j} \nu}(\xi) = \frac{1}{T_j} \sum_{a \in A_j} Y_a(j,\xi).  
\end{align}

By direct calculation, we have 
\begin{lemma}\label{I X lemma}
For each $j \in \ZZ_{j \geq 0}$, $a \in A_{j}$, $b \in B_{j+1,a}$, and $\xi \in \RR$, we have 
\begin{align}\label{I bound}
|I(a,b,j,\xi)| &\leq \min\cbr{1,N_{j+1}/|\xi|)}, \\
\label{X bound}
|X_a(j,\xi)| &\leq 2\min\cbr{1,N_{j+1}/|\xi|)}. 
\end{align}
\end{lemma}
Define $g:[0,\infty) \to [0,\infty)$ by 
$$
g(x) = (1+x)^{-1/2} + \sup\cbr{ |\widehat{\nu}(tx)| : t \in \RR, |t| \geq 1 }
$$
The following properties of $g$ are straightforward to verify. 
\begin{lemma} \label{g lemma}\hspace{1pt}
\begin{enumerate}[(a)]
	\item $g$ is non-increasing 
	\item For all $\xi \in \RR$, $|\widehat{\nu}(\xi)| \leq g(|\xi|)$. 
	\item For all $0 \leq \beta \leq 1$, 
	$\sup_{\xi \in \RR} |\widehat{\nu}(\xi)| (1+|\xi|)^{\beta/2} < \infty$ 
	if and only if $\sup_{\xi \in \RR} g(|\xi|) (1+|\xi|)^{\beta/2} < \infty$. 
\end{enumerate}
\end{lemma}
Recall that $\supp(\nu)$ is compact and does not contain $0$.  
Choose a 
Schwartz 
function $\phi:\RR \to \CC$ such that $\phi(y) = 1/y$ for all $y \in \supp(\nu)$. 
%
\begin{lemma}\label{phi nu decay lemma} 
For all $x \in \RR$, 
$|(\widehat{\phi \nu})(x)| \lesssim g(\tfrac{1}{2}|x|).$ 
\end{lemma}
\begin{proof}
Write
\begin{align*}
|(\widehat{\phi \nu})(x)| = |(\widehat{\phi} \ast \widehat{\nu})(x)| \leq \int_{\RR} |\widehat{\phi}(y)| \widehat{\nu}(|x-y|) dy. 
\end{align*}
Bound the integral above by the sum of the integrals over $R_1 = \cbr{y: \frac{1}{2}|x| \leq |x-y|}$ and $R_2 = \cbr{y: \frac{1}{2}|x| \leq |y|}$. 
For the integral over $R_1$, use (a) and (b) of Lemma \ref{g lemma}. For the integral over $R_2$, note $|\widehat{\phi}(y)|^{1/2} \lesssim g(|y|)$ for all $y \in \RR$ (because $\widehat{\phi}$ is Schwartz), then use (a) of Lemma \ref{g lemma}. 
%
\end{proof}

\begin{lemma}\label{J Y lemma}
There is a constant $C_0 > 0$ such that for each $j \in \ZZ_{\geq 0}$, $a \in A_{j}$, $b \in B_{j+1,a}$, and $\xi \in \RR$, we have 
\begin{align}
\label{J bound}
|J(a,b,j,\xi)| &\leq C_0 g(\tfrac{1}{2}|\xi|) \min\cbr{1,N_{j+1}/|\xi|)},  
\\
\label{Y bound}
|Y_a(j,\xi)| &\leq 2 C_0 g(\tfrac{1}{2}|\xi|)  \min\cbr{1,N_{j+1}/|\xi|)}. 
\end{align}
\end{lemma}
\begin{proof}
Note \eqref{Y bound} is immediate from \eqref{J bound}. 
Integrating $y$ in \eqref{J defn} shows that 
\begin{align*}
J(a,b,j,\xi) = \int_{[0,1]} \widehat{\nu}((\xi / N_{j+1})(aN_{j+1}+bN_{j+1}+x)) dx  
\end{align*}
Since $|(\xi / N_{j+1})(aN_{j+1}+bN_{j+1}+x)| \geq |\xi|$ for each $x \in [0,1]$, (a) and (b) of Lemma \ref{g lemma} give  
\begin{align*}
|J(a,b,j,\xi)| \leq g(|\xi|).
\end{align*}
On the other hand, integrating $x$ in \eqref{J defn} shows that 
\begin{align*}
J(a,b,j,\xi) = \int_{\RR}  e^{-2 \pi i (\xi / N_{j+1})(aN_{j+1}+bN_{j+1})y} \frac{e^{-2 \pi i (\xi / N_{j+1}) y} - 1}{-2 \pi i (\xi / N_{j+1}) y}  d\nu(y). 
\end{align*}
After multiplying by $-2 \pi i (\xi / N_{j+1})$, using that $\phi(y) = 1/y$ for all $y \in \supp(\nu)$, and integrating $y$,  
we see that $-2 \pi i (\xi / N_{j+1}) J(a,b,j,\xi)$ is  
$$
= \widehat{\phi \nu}((\xi / N_{j+1})(aN_{j+1}+bN_{j+1}+1)) - \widehat{\phi \nu}((\xi / N_{j+1})(aN_{j+1}+bN_{j+1})). 
$$
Since $|(\xi / N_{j+1})(aN_{j+1}+bN_{j+1}+x)| \geq |\xi|$ for each $x \in [0,1]$, Lemma \ref{phi nu decay lemma} and (a) of Lemma \ref{g lemma} give 
\begin{align*}
|J(a,b,j,\xi)| \lesssim \frac{1}{\pi} \left(\frac{N_{j+1}}{|\xi|}\right) g\left(\frac{1}{2}|\xi|\right). 
\end{align*}
\end{proof}

We need the following fact about averages over random subsets. 

\begin{lemma}\label{mean zero}
Let $t \leq n$ be positive integers. Let $A$ be a finite set of size $n$, and let $F: A \to \CC$. Let $\mathcal{B}_t$ be the collection of all size $t$ subsets of $A$, and let $B$ be a set chosen uniformly at random from $\mathcal{B}_t$. 
Then 
\begin{align*}
\EE\rbr{  \frac{1}{t} \sum_{x \in B} F(x)   } = \frac{1}{n} \sum_{x \in A} F(x). 
\end{align*}
\end{lemma}
\begin{proof}
There are $\binom{n}{t}$ sets in $\mathcal{B}_t$. For each $x \in A$, there are $\binom{n-1}{t-1}$ sets in $\mathcal{B}_t$ that contain $x$. Therefore 
\begin{align*}
\EE\rbr{  \frac{1}{t} \sum_{x \in B} F(x)   } = \frac{1}{{\textstyle \binom{n}{t}} \cdot t } \sum_{B \in \mathcal{B}_t}  \sum_{x \in B} F(x) = \frac{1}{{\textstyle \binom{n}{t}} \cdot t } \sum_{x \in A} \binom{n-1}{t-1} F(x) = \frac{1}{n} \sum_{x \in A} F(x). 
\end{align*}
\end{proof}

We need the following version of Hoeffding's inequality for complex-valued random variables. 

\begin{lemma}\label{bern}
Suppose $Z_1,\ldots,Z_t$ are independent complex-valued random variables satisfying $\EE(Z_i)=0$ and $|Z_i| \leq c$ for $i=1,\ldots,t$, where $c$ is a positive constant. For all $u > 0$,  
\begin{align*}
\PP\rbr{ \abs{\frac{1}{t} \sum_{i=1}^{t} Z_i} \geq c u } \leq 4 \exp\rbr{ -\frac{1}{4} t u^2 }. 
\end{align*}
\end{lemma}
\begin{proof}
Apply the standard Hoeffding inequality to the real and imaginary parts of $Z_1,\ldots,Z_t$. 
\end{proof}

\begin{lemma}\label{choice tele lemma}
Define $d_0$ by 
$
1/d_0 = \max\cbr{  \diam\rbr{\supp(\mu)},    \diam\rbr{\supp(\mu \cdot \nu)}   }. 
$
Fix a real number $\zeta_0$ such that 
\begin{align*}
\zeta_0 > \sum_{k \in \ZZ} \frac{2}{1+d_0^2|k|^2}.
\end{align*}
It is possible to choose the sets $B_{j+1,a}$ such that, 
for every $j \in \ZZ_{\geq 0}$ 
and every $\xi \in d_0 \ZZ = \cbr{d_0 k : k \in \ZZ}$, 
\begin{align}\label{tele ineq mu} 
|\widehat{\mu_{j+1}}(\xi) - \widehat{\mu_j}(\xi)| 
&< 2 T_j^{-1/2} \ln^{1/2}(4\zeta_0(1 + |\xi|^2)) \min\cbr{1,N_{j+1}/|\xi|},  
\\
\label{tele ineq m} 
|\widehat{\mu_{j+1} \nu}(\xi) - \widehat{\mu_j \nu}(\xi)| 
&< 2 C_0 T_j^{-1/2} \ln^{1/2}(4\zeta_0(1 + |\xi|^2)) g(\tfrac{1}{2}|\xi|) \min\cbr{1,N_{j+1}/|\xi|}, 
\end{align} 
where $C_0$ is the constant from Lemma \ref{J Y lemma}. 
\end{lemma}
\begin{proof}
Fix $j \in \ZZ_{\geq 0}$ and assume a set $A_j \subseteq [1,2)$ satisfying $|A_j|=T_j$ is given. 
To simplify notation in what follows, we write $N=N_{j+1}$, $n=n_{j+1}$, and $t = t_{j+1}$.  
For each $a \in A_j$, suppose we choose $B_{j+1,a}$ independently and uniformly at random from the collection of all size $t$ subsets of $N^{-1}[n]$. 
Now fix $\xi \in \RR$. 
Then 
$\cbr{X_a(j,\xi): a \in A_j }$ is a set of independent complex-valued random variables, 
and the same is true of $\cbr{Y_a(j,\xi): a \in A_j }$. 
%
Moreover, for each $a \in A_j$, we find that $\EE(X_a(j,\xi))=0$ by applying Lemma \ref{mean zero} with $F(b) = I(a,b,j,\xi)$ for $b \in N^{-1}[n]$. Likewise, we find that $\EE(Y_a(j,\xi))=0$ by applying Lemma \ref{mean zero} with $F(b) = J(a,b,j,\xi)$.
We also have the bounds on $|X_a(j,\xi)|$ and $|Y_a(j,\xi)|$ from Lemma \ref{I X lemma} and Lemma \ref{J Y lemma}, respectively. 
%
%
Define 
\begin{align*}
u_{j,\xi} &= \sqrt{ T_j^{-1} \ln\rbr{4\zeta_0(1+|\xi|^2)} }
\end{align*}
Let $E^1(\xi)$ be the event that $\abs{\frac{1}{T_j} \sum_{a \in A_j} X_a(j,\xi)} < 2 u_{j,\xi} \min\cbr{1,N/|\xi|}$. 
Let $E^2(\xi)$ be the event that $\abs{\frac{1}{T_j} \sum_{a \in A_j} Y_a(j,\xi)} < 2 C_0 u_{j,\xi} g(\frac{1}{2}|\xi|) \min\cbr{1,N/|\xi|}$. 
For each $\xi \in \RR$, Lemma \ref{bern} implies  
\begin{align*}
\PP\rbr{(E^i(\xi))^c} &\leq 4 \exp\rbr{-\frac{1}{4} T_j u_{j,\xi}^2} = \frac{1}{\zeta_0(1+|\xi|^2)}.  
\end{align*}
Therefore, by De Morgan's laws and countable subaddivity,
 the probability that both $E^1(d_0 k)$ and $E^2(d_0 k)$ hold for all $k \in \ZZ$ is 
 $$
 \geq 1 - \sum_{k \in \ZZ} \rbr{\PP\rbr{(E^1(d_0 k))^c} + \PP\rbr{(E^2(d_0 k))^c}} 
\geq 1 - \sum_{k \in \ZZ} \frac{2}{\zeta_0(1+d_0^2|k|^2)} 
> 0.  
 $$
In light of \eqref{mu equation} and \eqref{m equation}, this implies there is some choice of the sets $B_{j+1,a}$ $(a \in A_j)$ such that \eqref{tele ineq mu} and \eqref{tele ineq m} hold for every $\xi \in d_0 \ZZ$. 
\end{proof}

\begin{lemma}\label{mu mu nu decay lemma}
With the sets $B_{j+1,a}$ chosen as in Lemma \ref{choice tele lemma}, 
\begin{align}
\label{mu decay}
|\widehat{\mu}(\xi)| &\lesssim (1+|\xi|)^{-\alpha/2} \ln^{1/2}(4\zeta_0(1+|\xi|^2)) \quad \forall \xi \in \RR,  
\\
\label{m decay}
|\widehat{\mu \cdot \nu}(\xi)| &\lesssim (1+|\xi|)^{-\alpha/2} g\left(\frac{1}{2}|\xi|\right) \ln^{1/2}(4\zeta_0(1+|\xi|^2)) \quad \forall \xi \in \RR. 
\end{align}.
\end{lemma}
\begin{proof}
We prove only \eqref{m decay}, as the proof of \eqref{mu decay} is similar and simpler. We begin by making two reductions. 

First, by a standard argument (see Kahane \cite[pp.252-253]{kahane-book}), 
we only need to prove \eqref{m decay} for $\xi = d_0 k \in d_0 \ZZ$.  
For the second reduction, note that for every $0 \neq \xi \in \RR$, we have 
\begin{align*}
\widehat{{\mu_{0} \cdot \nu}}(\xi)  
&= \int_{\RR} \int_{1}^{2} e^{-2\pi i x y \xi} dx d\nu(y) 
=  \int_{\RR} \frac{e^{-2\pi i y \xi} - e^{-2\pi i y (2\xi)}}{-2 \pi i y \xi} d\nu(y) 
\\
&= \int_{\RR} \frac{e^{-2\pi i y \xi} - e^{-2\pi i y (2\xi)}}{-2 \pi i \xi} \phi(y) d\nu(y) 
= \frac{1}{-2 \pi i \xi} \rbr{ \widehat{\phi \nu}(\xi) - \widehat{\phi \nu}(2\xi) }. 
\end{align*}
By Lemma \ref{phi nu decay lemma} and (a) of Lemma \ref{g lemma}, $|\widehat{{\mu_{0} \cdot \nu}}(\xi)| \lesssim (1+|\xi|)^{-1} g\left(\frac{1}{2}|\xi|\right)$ for all $0 \neq \xi \in \RR$. The same inequality holds when $\xi=0$ by direct calculation. Therefore, by the triangle inequality, we just need to prove \eqref{m decay} with the left-hand side replaced by $|\widehat{\mu \cdot \nu}(\xi) - \widehat{{\mu_{0} \cdot \nu}}(\xi)|$. 

Lemma \ref{weak convergence} says $\mu_j \to \mu$ weakly. 
Then the dominated convergence theorem shows that ${\mu_{j} \cdot \nu} \to {\mu \cdot \nu}$ weakly. 
Therefore, for every $\xi \in \RR$, $\widehat{\mu \cdot \nu}(\xi) = \lim_{j \to \infty} \widehat{\mu_{j} \cdot \nu}(\xi)$, hence 
\begin{align}\label{sum above}
|\widehat{\mu \cdot \nu}(\xi) - \widehat{\mu_{0} \cdot \nu}(\xi)| \leq \sum_{j=0}^{\infty} |\widehat{\mu_{j+1} \cdot \nu}(\xi) - \widehat{\mu_{j} \cdot \nu}(\xi)|.  
\end{align} 
If $\xi = 0$, each term of the sum 
in \eqref{sum above} 
is zero, by direct calculation. 
Now assume $\xi = d_0 k \in d_0 \ZZ$, $\xi \neq 0$. 
By Lemma \ref{choice tele lemma}, the sum 
in \eqref{sum above} 
is 
\begin{align*}
&\leq 2 C_0 g(\tfrac{1}{2}|\xi|) \ln^{1/2}(4\zeta_0(1+|\xi|^2)) 
\rbr{ 
\sum_{j: N_{j+1} > |\xi|} T_j^{-1/2} 
+ \sum_{j: N_{j+1} \leq |\xi|}  T_j^{-1/2} \frac{N_{j+1}}{|\xi|}
}.
\end{align*}

To estimate the last two sums, recall that $2 \leq n_{j} \leq n_{\ast}$, $N_j = n_1 \cdots n_j$, $N_0=T_0=1$, and $T_{j} \approx N_{j}^{\alpha}$ for all $j \in \ZZ_{> 0}$. 
Thus the first sum is 
\begin{align*}
\approx |\xi|^{-\alpha/2} \sum_{j: N_{j+1} > |\xi|} (N_{j+1}/|\xi|)^{-\alpha/2} 
\leq  |\xi|^{-\alpha/2} \sum_{k=0}^{\infty} 2^{-k \alpha/2} 
\lesssim |\xi|^{-\alpha/2}. 
\end{align*}
And the second sum is 
\begin{align*}
\approx |\xi|^{-\alpha/2}  \sum_{j: N_{j+1} \leq |\xi|} (N_{j+1}/|\xi|)^{1-\alpha/2} 
\leq |\xi|^{-\alpha/2} \sum_{k=0}^{\infty} 2^{-k(1-\alpha/2)}
\lesssim |\xi|^{-\alpha/2}. 
\end{align*}
\end{proof}

\begin{lemma}
With the sets $B_{j+1,a}$ chosen as in Lemma \ref{choice tele lemma}, $\dim_F(\mu) \geq \alpha$ and $\dim_F(\mu \cdot \nu) \geq \min\cbr{1,\dim_F(\mu) + \dim_F(\nu)}$. 
\end{lemma}
\begin{proof}
Lemma \ref{mu mu nu decay lemma} 
implies $\dim_{F}(\mu) \geq \alpha$.  
Lemma \ref{mu mu nu decay lemma} and (c) of Lemma \ref{g lemma} imply $\dim_F(\mu \cdot \nu) \geq \dim_F(\mu) + \dim_F(\nu)$. 
\end{proof}

\bibliographystyle{amsplain}
\bibliography{bib_file}

\end{document}